\documentclass[12pt]{article}

\usepackage[a4paper,margin=30mm]{geometry}

\usepackage{amsmath,amssymb,amsthm,mathtools}
\usepackage{graphicx}
\usepackage{tikz}
\usepackage{xcolor}
\usepackage{hyperref}
\newtheorem{theorem}{Theorem}
\newtheorem{lemma}{Lemma}

\newtheorem{corollary}{Corollary}
\newtheorem{remark}{Remark}
\newtheorem{example}{Example}

\newcommand{\R}{\mathbb{R}}
\newcommand{\N}{\mathbb{N}}
\newcommand{\SOS}{\Sigma[x]}
\newcommand{\QM}{\mathcal{Q}}
\newcommand{\K}{K}

\def\revise#1{\textcolor{orange}{#1}}

\title{Convergence rate of the moment-SOS hierarchy for univariate polynomial optimization}
\author{Didier Henrion\footnote{LAAS-CNRS, Université de Toulouse, France, and Faculty of Electrical Engineering, Czech Technical University in Prague, Czechia.} and Mohab Safey El Din\footnote{Sorbonne Université, CNRS, LIP6, Paris, France.}}
\date{\today}

\begin{document}

\maketitle

\begin{abstract}
	We study the convergence rate of the moment-SOS (sum-of-squares) hierarchy
	for polynomial optimization problems (POPs) on a bounded subset of the real line described by
	arbitrary polynomial inequalities. We prove that, for every fixed univariate POP, the relaxation error is bounded by
	$O(1/r^2)$, where $r$ is the relaxation order. In particular, boundary
	degeneracies in the polynomial description of the feasible set affect the
	constant but not the convergence exponent. The proof combines the structure
	of finitely generated univariate quadratic modules with a Chebyshev
	polynomial construction that approximately recovers the natural generators of
	the feasible set while controlling the degree of the certificate. We also give an
	elementary degree-four example for which the relaxation error is exactly
	$1/(2r(r-1))$, showing that the quadratic rate is optimal. Equivalent
	reformulations connect this example to a cubic univariate problem and to a
	bivariate POP whose feasible set has a cusp
	singularity at the minimizer.
\end{abstract}

\section{Introduction}

\paragraph{The moment-SOS hierarchy for POP.}

A polynomial optimization problem (POP) consists of minimizing globally a given polynomial subject to polynomial inequalities:
\begin{equation}\label{eq:pop}
	f^* := \min_x f(x) \ \text{s.t.}\ x \in K:=\{x \in \R^n : p_i(x) \geq 0, \ i=1,\ldots,n_p\}
\end{equation}
where $f, p_1, \ldots, p_{n_p} \subset \R[x]$ are given polynomials.
It is assumed that $K$ is non-empty and compact, so that a minimizer exists. In full generality, \eqref{eq:pop} is a difficult non-convex problem, potentially with several minimizers. It includes as particular cases all combinatorial optimization problems. The moment-SOS (sums of squares) hierarchy, as originally proposed by Lasserre \cite{L01}, is a computationally efficient approach for polynomial optimization. It consists of constructing a family of convex relaxations of increasing size, generating a sequence of lower bounds on the minimum. Under some conditions, solving a sufficiently large convex relaxation is equivalent to solving the original POP. 
Let $\SOS \subset \R[x]$ denote the cone of polynomials that can be written as sums of squares (SOS), and let $p_0(x) \equiv 1$ for notational convenience. The dual SOS formulation of the moment-SOS hierarchy consists of solving the semidefinite optimization problems
\begin{equation}\label{eq:sos}
	f_r := \sup_{v \in \R} v \ \text{s.t.}\ f-v \in \QM_{{r}}(p)
\end{equation}
in the truncated quadratic module
\begin{equation}\label{eq:tqm}
	\QM_r(p):= \left\{\sum_{i=0}^{n_p} \sigma_i p_i, \ \sigma_i \in \SOS, \ \text{deg}(\sigma_i p_i) \leq 2r\right\} \subset \R[x]
\end{equation}
indexed by $r \in \N$ such that $2r \geq \max\{\text{deg}\:f, \text{deg}\:p_1, \ldots, \text{deg}\:p_{n_p}\}$. This problem is called the SOS relaxation of order $r$. See \cite{H23} for an elementary tutorial in the moment-SOS hierarchy for polynomial optimization, \cite{N23,T24} for recent overviews, and \cite{M08,S24} for the underlying real algebraic geometry.

\paragraph{Relaxation gap and convergence.}
An SOS relaxation is characterized by its relaxation gap, or error
\[
	e_r := f^*-f_r \geq 0
\]
which is non-negative by construction.
In \cite{L01}, asymptotic convergence was shown:
\[
	e_r \geq e_{r+1} \geq \cdots \geq \lim_{r\to\infty} e_r = 0
\]
by using Putinar's Positivstellensatz \cite{P93} which states that any given polynomial strictly positive on $K$ belongs to the quadratic module $\QM_r(p)$ if $r$ is chosen sufficiently large.
In \cite{N14}, generic finite convergence was shown: $e_{r^*} = 0$ for a finite $r^*$ depending on the POP data. Non-finite convergence may occur, i.e. $e_r > 0$ for all finite $r$, but only for specific POPs whose data $f, p_1, \ldots, p_{n_p}$ lie on a low-dimensional manifold. In this paper, we are precisely interested in these instances, and we want to quantify the speed at which the error sequence $e_r$ converges to zero. 

{\paragraph{Convergence rates.}
Quantifying the speed of this convergence amounts to bounding the degree of the
SOS multipliers needed to certify positivity. All the estimates below concern
the quadratic module \eqref{eq:tqm} and the Putinar Positivstellensatz. Sharper rates are typically available for the preordering, which also allows products of the $p_i$ as
multipliers in the Schm\"udgen Positivstellensatz, but the corresponding relaxations involve $2^{n_p}$ instead of
$n_p+1$ semidefinite blocks, so that the Putinar-type rates are the practically
relevant ones. The first degree bound, exponential in the inverse of the
accuracy, was given in \cite[Thm.~6]{NS07} and yields only an
inverse-logarithmic rate \cite[Thm.~8]{NS07}. Whether this exponential
dependence could be removed stayed open for more than a decade, until
\cite{BM23} obtained the first general polynomial degree bound, governed by a
{\L}ojasiewicz exponent attached to the description of the set; see
\cite{BMP25} for refinements and \cite{HVZ26} for a recent sharpening based on
a lift-and-project construction. Much stronger rates are known on distinguished
sets, through the polynomial kernel method and its Christoffel-Darboux
reformulation: a rate $O(1/r^2)$ on the sphere for homogeneous $f$ \cite{FF21},
extended to arbitrary $f$ in \cite{BL25}. On the hypercube $[-1,1]^n$,
\cite[Thm.~3, Cor.~15]{BS24} prove a rate $O(1/r)$, whereas the only available
lower bound is of order $1/r^8$ \cite[Thm.~4, Cor.~16]{BS24}; a new family of
squared approximation kernels has recently improved the upper rate to
$O(\log^3 r/r^2)$ \cite{GDV26}, leaving the lower bound as the loose end. This
illustrates the fact that the study of the convergence rate of the moment-SOS
hierarchy is currently an active research topic, see the recent overviews
\cite{LS26,STL26}.}

\paragraph{Univariate POP.}
In this paper, we study the convergence rate of the moment-SOS hierarchy in the particular case of a univariate POP, i.e. $n=1$ in \eqref{eq:pop}.
Unless the degree $d$ of $f$ is very large, the univariate POP \eqref{eq:pop} is not challenging as an optimization problem. Indeed, if $z_1, \dots z_{d-1}$ denote the roots of the derivative of $f$, then 
{a global minimizer on \(K\) is attained among the finitely many
points $\bigl(\{z_1,\ldots,z_{d-1}\}\cap K\bigr)
	\cup
	\{a_1,b_1,\ldots,a_{m-1},b_{m-1}\}$},
and there is no need for a sophisticated machinery such as the moment-SOS hierarchy. However, studying the worst case convergence rate  in this simple univariate setting provides significant insight on the weaknesses of the hierarchy. Indeed, many challenging multivariate POPs with specific degeneracies reduce to univariate POPs, see e.g. \cite{HLM25}.

\paragraph{Contributions}
In the univariate POP setup \eqref{eq:pop}, our main contributions are twofold. First, we prove that there is a constant $C$ (depending on the POP data) such that
\[
	0 \leq   e_r \leq \frac{C}{r^2}
\]
for sufficiently large $r$. It means that the moment-SOS hierarchy cannot converge slower than $O(1/r^2)$.
Second, we describe an elementary degree 4 example achieving exactly this convergence rate, i.e. our convergence rate is tight. Our example is a simplification of a well-known example due originally to Stengle \cite{S96}, recently solved in rational arithmetic \cite{H25}. Finally, we relate our example, through equivalent reformulations, to an elementary cubic POP with two constraints, and to a bivariate POP whose feasible set has a cusp singularity at the minimizer \cite[Ex.~9.4.6(3)]{M08}, \cite[Ex.~A.3]{BM25}, \cite{K25}, thereby illustrating how multivariate POPs with specific degeneracies reduce to the univariate setting studied here.

\section{Natural Generators}

Since we assume that the set $K\subset \R$ in \eqref{eq:pop} is bounded and semi-algebraic, it can be expressed as a finite union of bounded intervals
\[
	K = [a_1,b_1] \cup [a_2,b_2] \cup \cdots \cup [a_{m-1},b_{m-1}]
\]
for some reals $a_i \leq b_i < a_{i+1}$ and integer $m > 1$. Associated to this representation is the natural choice of generators defined in \cite[\S 2.3]{KM02}, see also \cite[\S 6.1.2]{S24}:
\[
	g_1(x) := x-a_1, \ g_2(x) := (x-b_1)(x-a_2), \ldots, g_m(x) = b_{m-1}-x
\]
such that
\[
	K = \{x \in \R : g_i(x) \geq 0, \ i=1,\ldots,m\}.
\]

\begin{example}
	Let $K=[-1,0] \cup [1,2]$. Then $$g_1(x)=x+1,\quad g_2(x)=x(x-1), \quad g_3(x)=2-x.$$
	For $f(x) = g_1(x)g_2(x) = x^3-x$ it holds $f^* = \min_K f = 0$. The SOS relaxation \eqref{eq:sos} of order $r=2$ yields $e_2 = 6-4\sqrt{2} \approx 0.34315$ with
	\[
		f(x)+6-4\sqrt{2} = (x+\sqrt{2}-2)^2 g_1(x) + (\sqrt{2}-1)^2 g_2(x).
	\]
	The SOS relaxation \eqref{eq:sos} of order $r=3$ yields $e_3 = \tfrac18 = 0.125$ with
	\[
		f(x)+\tfrac18 = \tfrac{1}{12}(x^2-\tfrac52 x+1)^2 g_1(x) + \tfrac14 (x+\tfrac32)^2 g_2(x) +   \tfrac{1}{12}(x^2+\tfrac12 x-\tfrac12)^2 g_3(x).
	\]
	And finally the SOS relaxation \eqref{eq:sos} of order $r=4$ yields $e_4 = 0$ with
	\begin{align*}
		f(x) = & x^2(x-1)^2(\tfrac{1}{20}(x-\tfrac32)^2+\tfrac{3}{16})g_1(x) + \\
		       & (x+1)^2(\tfrac{1}{10}(x-3)^2+\tfrac{1}{10}) g_2(x) +
		\tfrac{1}{20}x^2(x^2-1)^2 g_3(x).
	\end{align*}
	In this case the moment-SOS hierarchy converges at the 4th relaxation order:
	\[
		e_2 =  6-4\sqrt{2} > e_3 =  \frac18 > e_4 = 0.
	\]
	This example was studied in \cite[Ex. 4.4(a)]{KMS05} as a limit case of the parametric problem
	\[
		K_N = [-1,0] \cup [1,N], \quad N \geq 2.
	\]
	It was shown that $g_1 g_2 \in \QM_{{r_N}}(g) \setminus \QM_{{r_N-1}}(g)$ with truncation degree $r_N$ not uniformly bounded in $N$, i.e.  $\lim_{N\to\infty} r_N = \infty$. In other words, the moment-SOS hierarchy has finite convergence when $K_N$ is bounded, but the convergence can be arbitrarily slow. Note that an affine change of variable shows that the same phenomenon appears with the uniformly bounded set $[0, \tfrac{1}{N}] \cup [\tfrac{2}{N},1]$. A similar parametric univariate example with arbitrarily slow convergence was studied in \cite{HLM25}.

	Note also that this arbitrary slow convergence of the moment-SOS hierarchy is a property of the chosen natural generator presentation, not merely of the set and objective. The same set can be described by the
	single quartic generator:
	\[
		K_N=  \{ x\in \R :  h_N(x):=(x+1)x(x-1)(N-x)\geq 0 \}
	\]
	and the moment-SOS hierarchy then admits a uniform order-\(3\) certificate for every \(N\geq 2\), namely
	\[
		f(x)
		=
		\frac{f(x)^2}{N(N^2-1)}
		+
		\frac{\left(x+\frac{N}{2}\right)^2+\frac{3N^2-4}{4}}{N(N^2-1)}\,h_N(x).
	\]
    Indeed, $h_N$ is the product of the natural generators of the
	compact set $K_N$, which has no isolated point, so that $\QM(h_N)$ already
	contains every polynomial nonnegative on $K_N$
	\cite[Ex.~6.5.17(3)]{S24}.
\end{example}

Note that \eqref{eq:tqm} is a finite-dimensional truncation of the quadratic module
\begin{equation*}
	\QM(p):= \left\{\sum_{i=0}^{n_p} \sigma_i p_i, \ \sigma_i \in \SOS \right\} \subset \R[x].
\end{equation*}

\begin{lemma}[Archimedean property]\label{lem:archi}
	There exists $R\geq 1$ such that $R^2-x^2 \in \QM(p)$.
\end{lemma}

\begin{proof}
	Since $K$ is a compact subset of the real line, we can use \cite[Thm.~7.1.2]{M08} or \cite[Prop.~5.5.14]{S24}, which establishes that $\QM(p)$ is Archimedean. The fact that $R\geq 1$ comes from \cite[Cor.~5.2.4]{M08}.
\end{proof}

\begin{lemma}[Markov-Luk\'acs]\label{lem:interval}
	Let $R>0$.
	If $q\in\R[x]$ is nonnegative on $[-R,R]$ and has degree $d$, then
	\begin{equation*}
  q\in \QM_{\lceil\tfrac{d}{2}\rceil}(R^2-x^2).
	\end{equation*}
\end{lemma}

\begin{proof}
	This is the classical Markov-Luk\'acs theorem, see \cite[Sec.
			2]{PR00} for a modern account.
  
	{For $q=0$ the result is trivial. If $d=0$, then $q$ is a
    nonnegative constant and the result is immediate. Hence assume
    $q\neq0$ and $d\geq1$.} Up to scaling by $1/R$, one can assume that $R = 1$ without loss of generality.
	Let $\mathsf{G}_m(q)$ be the Goursat transform of $q$ of degree $m$,
		i.e. $\mathsf{G}_m(q)(x) = (1+x)^m q\left( \frac{1-x}{1+x}\right )$.
		Observe that $2^m q
			= \mathsf{G}_m\left(\mathsf{G}_m(q)\right )$. 

	By Goursat's lemma \cite[Lemma 1]{PR00}, $q$ is nonnegative on $[-1,
				1]$ if and only if $\mathsf{G}_d(q)$ is nonnegative on $[0, +\infty[$.
	Moreover, it holds that $\deg G_d(q) \leq d := \deg q$.
	By \cite[Prop. 2]{PR00}, if $\mathsf{G}_d(q)$ is nonnegative on $[0,
		+\infty[$, then there exist two sums of squares $s_0$ and $s_1$, with
	$\deg(s_0) \leq d$  and $\deg(s_1) \leq d-1$, such that $\mathsf{G}_d(q) = s_0 +
		x s_1$.
	Now, we perform a Goursat transform of degree $d$ on both sides of
		this identity. We obtain
		\begin{align*}
			2^d q & =(1+x)^d s_0\left ( \frac{1-x}{1+x}\right ) + (1+x)^d \left (\frac{1-x}{1+x}\right ) s_1\left(\frac{1-x}{1+x}\right ) \\
			      & =\mathsf{G}_d(s_0) + (1-x)(1+x)^{d-1}s_1\left(\frac{1-x}{1+x}\right ).
		\end{align*}
	Note that $\mathsf{G}_d(s_0)$ has degree at most $d$. Since $s_0$
		(resp. $s_1$) is a sum of squares, there exist $a_1, \ldots, a_u$ (resp.
		$b_1, \ldots, b_v$) such that $s_0 = a_1^2 + \cdots + a_u^2$ (resp. $s_1 =
			b_1^2+\cdots +b_v^2$). Recall
		that we established that $\deg s_0 \leq d$ and $\deg s_1 \leq d-1$.
		We deduce that $\deg a_i \leq \frac{d}{2}$ for all $1 \leq i \leq u$ and $\deg b_i \leq \frac{d-1}{2}$ for all $1 \leq i \leq v$.
		When $d$ is even, we deduce that
		\[
			\mathsf{G}_d(s_0) = \sum_{i=1}^u \mathsf{G}_{d/ 2}(a_i)^2
		\]
		and that
		\begin{align*}
			(1+x)^{d-1}s_1\left(\frac{1-x}{1+x}\right ) & = (1+x)^{d-1}\sum_{i=1}^v  b_i^2\left(\frac{1-x}{1+x}\right ) \\
			                                            & = (1+x) \sum_{i=1}^v  \mathsf{G}_{(d-2) / 2}(b_i)^2
		\end{align*}
		and we conclude that
		\[
			2^d q = \sum_{i=1}^{\revise u}\mathsf{G}_{d/2}(a_i)^2 + {(1-x^2)} \sum_{i=1}^v  \mathsf{G}_{(d-2)/2}(b_i)^2
		\]
		which establishes that \[q \in \QM_{{d/2}}(1-x^2).\] 	

	When $d$ is odd, we deduce that
		\[
			\mathsf{G}_d(s_0) = (1+x) \sum_{i=1}^u \mathsf{G}_{(d-1)/2}(a_i)^2
		\]
		and that
		\begin{align*}
			(1+x)^{d-1}s_1\left(\frac{1-x}{1+x}\right ) & = \sum_{i=1}^v  \mathsf{G}_{(d-1) / 2}(b_i)^2
		\end{align*}
		To summarize, we have
		\begin{equation}\label{eq:oddcase}
			2^d q  = (1+x)\sum_{i=1}^u \mathsf{G}_{(d-1)/2}(a_i)^{\revise 2} + (1-x)\sum_{i=1}^v \mathsf{G}_{(d-1)/2}(b_i)^{\revise 2}.
		\end{equation}
 {Using
\[
	1+x=\frac{(1+x)^2}{2}+\frac{1-x^2}{2},
	\qquad
	1-x=\frac{(1-x)^2}{2}+\frac{1-x^2}{2},
\]
in \eqref{eq:oddcase}, and observing that all resulting terms have
degree at most $d+1$,
we obtain
\[
	q\in\QM_{(d+1)/2}(1-x^2).
\]}	
	%
\end{proof}

\begin{lemma}[Natural generator representation]\label{lem:natgen}
	Let $g=(g_1,\ldots,g_m)$ be the vector of natural generators of $\K$, and let $g_0\equiv 1$. Then
	\[
		\{q\in\R[x]:q\geq 0\text{ on }\K\} = \QM(g) = \left\{\sum_{i=0}^m\sigma_i g_i, \ \sigma_i \in \SOS\right\}.
	\]
\end{lemma}

\begin{proof}
	Every polynomial nonnegative on $\K$ belongs to the preordering generated
	by the natural generators $g_i$; this is the saturation result of
	\cite[Thm.~2.2]{KM02}, see also
	\cite[Prop.~2.7.3]{M08} and
	\cite[Prop.~6.1.3(a)]{S24}.
	Since $\K$ is bounded, the quadratic module $\QM(g)$ is closed
	under multiplication, and therefore coincides with this preordering, see \cite[Cor.~4.4]{S05},   \cite[Thm.~9.3.1]{M08}, and \cite[Prop.~6.5.8]{S24} applied to the affine line.
\end{proof}

\section{Upper Bound on Convergence Rate}

\subsection{Statement}

We consider the problem \eqref{eq:pop} for $n=1$:
\begin{equation*}
	f^* := \min_x f(x) \ \text{s.t.}\ x \in K:=\{x \in \R : p_i(x) \geq 0, \ i=1,\ldots,n_p\}
\end{equation*}
where $f$ and $p = (p_1, \ldots, p_{n_p})$ are in $\R[x]$.
Recall that we assume that $$\K=[a_1, b_1]\cup \cdots \cup [a_{m-1}, b_{m-1}]$$ is
bounded and non-empty.
\begin{theorem}[Quadratic convergence rate]
	\label{thm:main}
	There exist constants
	$C>0$ and $r_0\in\N$ such that
	\begin{equation}
		\label{eq:mainrate}
		0\leq e_r\leq \frac{C}{r^2},
		\qquad r\geq r_0.
	\end{equation}
\end{theorem}

The proof proceeds by separating the  {geometry of the feasible set}
from the {degree complexity of the given quadratic module}.
Let $q:=f-f^\star$ so that $q\geq 0$ on $\K$. As recalled by Lemma \ref{lem:natgen}, since we are in the univariate case, $q$ belongs to the quadratic module $\QM(g)$ of the natural generators of $\K$.
The difficulty is that the natural generators need not themselves
belong to the original quadratic module $\QM(p)$.

The proof overcomes this difficulty in two steps.

First, the structure theory of finitely generated univariate quadratic modules
shows that, for each natural generator $g_j$, some odd power belongs to the
original module, i.e. $g_i^{d_i}\in\QM(p)$ for $d_i$ odd.

Second, a Chebyshev polynomial construction converts such an odd-power
certificate into an approximate certificate for the generator itself:
for every $n\geq 1$, $g_i+C_i\varepsilon_n \in \QM_{O(n)}(p)$ with  $\varepsilon_n=O(n^{-2})$.
Substituting these approximate generators into the fixed natural generator
representation  gives $q+C\varepsilon_n \in \QM_{O(n)}(p)$.\\
Finally, choosing $n$ proportional to the relaxation order $r$ yields
$e_r=O(1/r^2)$.

\subsection{Proof}

Let
\[
	a
	=
	\operatorname{arcosh}3
	=
	\log\left(3+2\sqrt{2}\right)
	=
	2\log\left(1+\sqrt{2}\right)
	=
	2\operatorname{arsinh}1
	\approx
	1.76275
\]
and
\[
	\varepsilon_n
	=
	\frac{\cosh(a/n)-1}{2}
	=
	\sinh^2\!\left(\frac{a}{2n}\right)
	=
	\frac{1}{4}
	\left(
	\left(1+\sqrt{2}\right)^{1/n}
	-
	\left(\sqrt{2}-1\right)^{1/n}
	\right)^{2},
\]
for every integer $n\geq 1$: in particular
$\varepsilon_1=1$ and $\varepsilon_2=\frac{\sqrt{2}-1}{2}$.
Moreover
\[
	\lim_{n\to\infty}n^2\varepsilon_n
	=
	\frac{a^2}{4}
	=
	\log^2\left(1+\sqrt{2}\right)
	\approx
	0.77682.
\]

\begin{lemma}
	\label{lem:epsbound}
	For every $n\geq 1$, it holds $0<\varepsilon_n\leq \frac{1}{n^2}$.
\end{lemma}

\begin{proof}
	Because $a>0$, it holds
	\[
		0 < \cosh(a/n)-1
		=
		\sum_{k\geq 1}\frac{a^{2k}}{(2k)!n^{2k}}
		\leq
		\frac{1}{n^2}
		\sum_{k\geq 1}\frac{a^{2k}}{(2k)!}
		=
		\frac{\cosh a-1}{n^2}
		=
		\frac{2}{n^2}.
	\]
\end{proof}

Let $T_n$ be the Chebyshev polynomial of the first kind. Recall that these are defined by the recurrence $T_{n+1}(t) = 2t T_n(t)-T_{n-1}(t)$, initialized with $T_0(t)=1$ and $T_1(t)=x$. We also have $T_n(\cos\theta)=\cos(n\theta)$.
Define
\[
	V_n(t)
	:=
	\frac{1-T_n(1-2t)}{2}.
\]
Since $V_n(0)=0$, the quotient
\[
	W_n(t):=\frac{V_n(t)}{t}
\]
is a polynomial of degree $n-1$.

\begin{lemma}[Odd-multiplicity desingularization]
	\label{lem:chebatom}
	Let $d\geq 1$ be odd and define
	\begin{equation}
		\label{eq:qnnu}
		q_{n,d}(t)
		:=
		\varepsilon_n
		+
		t(1-V_n(t)^{d-1}).
	\end{equation}
	Then
	\begin{equation}
		\label{eq:qnnonneg}
		q_{n,d}(t)\geq 0
		\qquad\text{for every }t\leq 1,
	\end{equation}
	and
	\begin{equation}
		\label{eq:chebidentity}
		t+\varepsilon_n
		=
		q_{n,d}(t)
		+
		t^d W_n(t)^{d-1}.
	\end{equation}
	Moreover,
	\begin{equation}
		\label{eq:chebdegree}
		\deg q_{n,d}
		\leq 1+n(d-1),
		\qquad
		\deg(t^d W_n^{d-1})
		\leq 1+n(d-1).
	\end{equation}
\end{lemma}

\begin{proof}
	Since $V_n(t)=tW_n(t)$, it holds that
	$t^d W_n(t)^{d-1}=tV_n(t)^{d-1}$ and
	\eqref{eq:chebidentity} follows immediately from
	\eqref{eq:qnnu}.
	It remains to prove \eqref{eq:qnnonneg}, i.e. $q_{n,d}(t)\geq 0$ for every $t\leq 1$.

	First let $0\leq t\leq 1$.  Then $1-2t\in[-1,1]$, so
	$-1\leq T_n(1-2t)\leq 1$ and hence $0\leq V_n(t)\leq 1$,
	which implies $1-V_n(t)^{d-1}\geq 0$,
	and therefore $q_{n,d}(t)\geq \varepsilon_n>0$.

	Second, let $t<0$ and set $z:=1-2t>1$.\\
	For $z>1$, $T_n(z)= \cosh(n\operatorname{arcosh}z)>1$,
	so that $$V_n(t)=(1-T_n(z))/2<0.$$
	We investigate the two complementary cases
	$V_n(t) \leq -1$ and $-1 < V_n(t) < 0$.

	If $V_n(t)\leq -1$, then $V_n(t)^{d-1}\geq 1$ because $d-1$ is even,
	so that $1-V_n(t)^{d-1}\leq 0$.
	As $t<0$, $t(1-V_n(t)^{d-1})\geq 0$,
	hence again $q_{n,d}(t)\geq\varepsilon_n$.

	Suppose finally that $-1 < V_n(t)<0$.  Then, since $d-1$ is even, it holds that
	$$0 < V_n(t)^{d-1} < 1 \quad \text{ and } \quad
		T_n(z)=1-2V_n(t)\leq 3=\cosh a.$$
	Because $\cosh$ is increasing on $[0,\infty)$, $\operatorname{arcosh}z\leq \frac{a}{n}$.
	Thus $z\leq \cosh(a/n)$,
	and consequently $$|t| =(z-1)/2 \leq (\cosh(a/n)-1)/2:=\varepsilon_n.$$
	Therefore $q_{n,d}(t)=\varepsilon_n-|t|(1-V_n(t)^{d-1}) \geq \varepsilon_n-|t| \geq 0$.
	This ends the proof of \eqref{eq:qnnonneg}.

	Finally, note that, by definition of $V_n$, it holds that $\deg V_n=\deg T_n
		= n$ and $\deg W_n=\left (\deg V_n \right )- 1 = n-1$, which establishes
	\eqref{eq:chebdegree}.
\end{proof}

\begin{figure}[htbp]
	\centering
	\begin{tabular}{cc}
		\includegraphics[width=.48\textwidth]{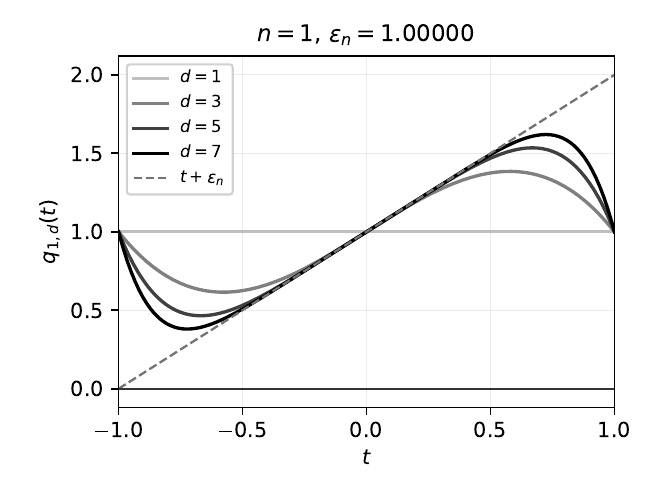} &
		\includegraphics[width=.48\textwidth]{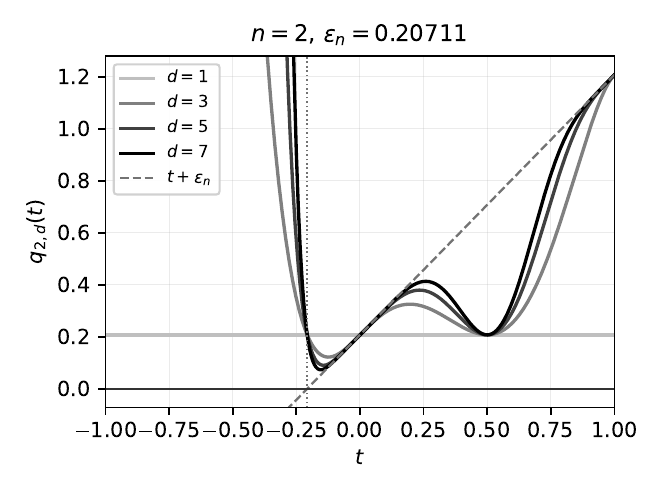}   \\
		\includegraphics[width=.48\textwidth]{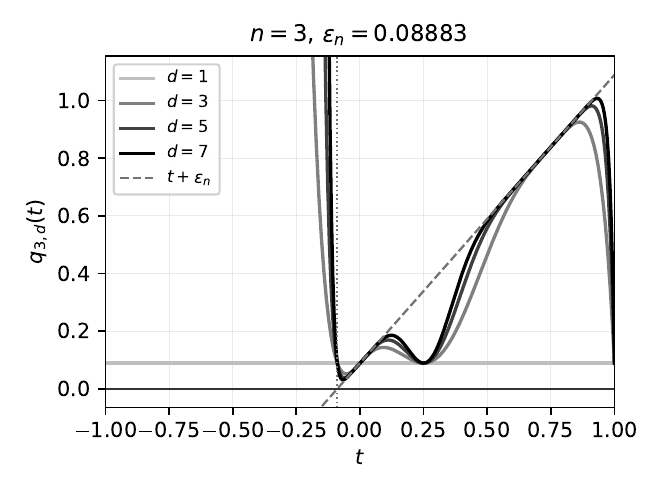} &
		\includegraphics[width=.48\textwidth]{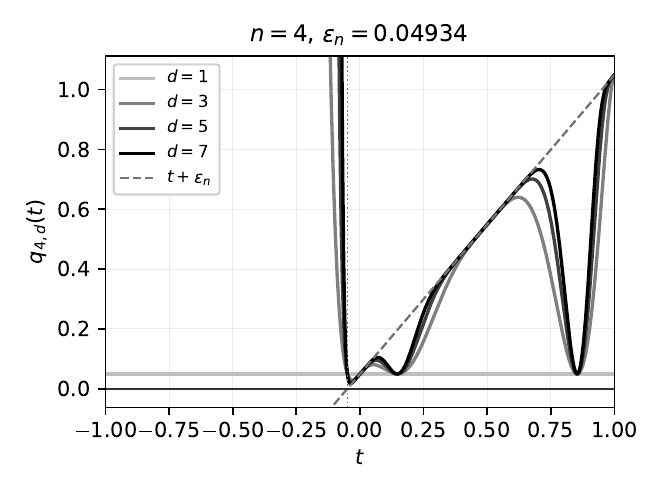}
	\end{tabular}
	\caption{Graphs of the polynomials $q_{n,d}$ of
		Lemma~\ref{lem:chebatom} on $(-1,1)$, for $n=1,2,3,4$ and odd
		exponents $d=1,3,5,7$. On $[0,1]$ each graph oscillates between the
		constant floor $\varepsilon_n$, attained where $V_n(t)=1$, and the
		line $t+\varepsilon_n$ (dashed), attained where $V_n(t)=0$. On the
		interval $(-\varepsilon_n,0)$, delimited by the dotted vertical line
		at $t=-\varepsilon_n$, the graphs dip below $\varepsilon_n$ while
		remaining bounded from below by $t+\varepsilon_n\geq 0$, in
		accordance with the nonnegativity certificate \eqref{eq:qnnonneg}.
		For $t<-\varepsilon_n$ one has $V_n(t)<-1$, so the polynomials
		increase very steeply and leave the plotted range. The degenerate
		case $d=1$ yields the constant $q_{n,1}\equiv\varepsilon_n$.}
	\label{fig:qnd}
\end{figure}

Recall that $p = (p_1, \ldots, p_{n_p})$ is the vector of polynomials whose
simultaneous non-negativity defines $$\K = [a_1, b_1]\cup \cdots \cup [a_{m-1},
		b_{m-1}] .$$
Recall that we assume $\K$ be bounded and non-empty.

\begin{lemma}[Odd powers of the natural generators]
	\label{lem:oddpower}
	For every $i=1,\ldots,m$, there exists an odd integer
	$d_i\geq 1$ such that
	\begin{equation*}
		g_i^{d_i}\in\QM(p).
	\end{equation*}
\end{lemma}

\begin{proof}
	We repeatedly use the generalized natural-generator description of a finitely
	generated univariate quadratic module with nonempty bounded positivity
	set. By \cite[Cor.~4.6]{A12}, together with
	\cite[Defs.~4.1 and~4.4]{A12} (which we can apply because we assume
		that $\K$ is bounded and non-empty), every generalized natural generator
	belongs to $\QM(p)$, see also \cite[Thm.~2.26 and Cor.~2.28]{A08}.

	At the left boundary, \cite[Cor.~4.6]{A12} establishes that
	there exists an odd integer \(d_1\geq1\)
	such that   $g_1^{d_1}(x)=(x-a_1)^{d_1}$
	belongs to \(\QM(p)\).

	For each gap \((b_i,a_{i+1})\), \(i=1,\ldots,m-2\), using again
	\cite[Cor.~4.6]{A12}, we deduce that there exist odd integers
	\(\beta_i,\alpha_{i+1}\geq1\) such that $h_{i+1}(x) :=
		(x-b_i)^{\beta_i}(x-a_{i+1})^{\alpha_{i+1}} \in\QM(p)$. Set
	$d_{i+1}:=\max\{\beta_i,\alpha_{i+1}\}$. Since \(\beta_i\) and
	\(\alpha_{i+1}\) are odd, \(d_{i+1}\) is odd, and both $d_{i+1}-\beta_i$ and
	$d_{i+1}-\alpha_{i+1}$ are nonnegative even integers. Consequently,
	\[
		\begin{aligned}
			g_{i+1}(x)^{d_{i+1}}
			 & =
			\bigl((x-b_i)(x-a_{i+1})\bigr)^{d_{i+1}} \\
			 & =
			h_{i+1}(x)
			\left(
			(x-b_i)^{(d_{i+1}-\beta_i)/2}
			(x-a_{i+1})^{(d_{i+1}-\alpha_{i+1})/2}
			\right)^2.
		\end{aligned}
	\]
	A quadratic module is closed under multiplication by polynomial
	squares. Hence $g_{i+1}^{d_{i+1}}\in\QM(p)$, for $i=1,\ldots,m-2$.

	Finally, at the right boundary, \cite[Cor.~4.6]{A12} shows that there exists
	an odd integer \(d_m\geq1\) such that
	$g_m^{d_m}(x)=(b_{m-1}-x)^{d_m}\in\QM(p)$.
\end{proof}

\begin{example}
	Consider $p_1(x)=(1-x)^3(1+x)$,
	whose positivity set is $K=[-1,1]$.
	The natural generators are $g_1(x)=1+x$, $g_2(x)=1-x$. It holds
	\[
		g_1
		=
		\frac18g_1^2((g_2+1)^2+3)
		+\frac18p_1 \text{ which lies in }\QM(p)
	\]
	so \(d_1=1\).
	For the other natural generator,
	\[
		g_2^3
		=
		\frac12g_2^4+\frac12p_1 \text{ which also lies in }\QM(p)
	\]
	so \(d_2=3\).
	Moreover, \(g_2\notin\QM(p)\). Indeed, suppose that
	$g_2=\sigma_0+\sigma_1p_1$, $\sigma_0,\sigma_1\in\Sigma[x]$.
	Evaluating at \(x=1\) gives \(\sigma_0(1)=0\). Since \(\sigma_0\) is a
	sum of squares, its zero at \(x=1\) has even multiplicity, and hence
	$\sigma_0'(1)=0$. Since \(p_1\) has a zero of multiplicity 3 at \(x=1\), one also has $p_1(1)=p_1'(1)=0$. Differentiating the assumed identity at \(x=1\) would therefore give $g_2'(1)=0$, contrary to \(g_2'(1)=-1\). Thus \(g_2\notin\QM(p)\), and the smallest possible odd exponent is \(d_2=3\).
    The failure of $g_2$ to lie in $\QM(p)$ is an instance of the criterion
	\cite[Prop.~6.5.16]{S24}: on a non-singular affine curve, a finitely
	generated preordering with compact positivity set is saturated if and only if
	some generator vanishes to order exactly one at every boundary point, which
	fails here because $p_1$ vanishes at order $3$.
\end{example}

\begin{example}
	Consider $p_1(x)=-x^2$, whose positivity set is the singleton
	$K=\{0\}$.
	The natural generators  are $g_1(x)=x$, $g_2(x)=-x$.
	It holds that
	\[
		g_1^3=x^3
		=
		\left(\frac{x(x+1)}2\right)^2
		+
		\left(\frac{x-1}{2}\right)^2(-x^2) \in\QM(p)
	\]
	so $d_1=3$.
	Similarly,
	\[
		g_2^3=-x^3
		=
		\left(\frac{x(x-1)}2\right)^2
		+
		\left(\frac{x+1}{2}\right)^2(-x^2) \in\QM(p)
	\]
	so $d_2=3$.\\
	Neither \(x\) nor \(-x\) belongs to \(\QM(-x^2)\). For instance, if $x=\sigma_0-\sigma_1x^2$,
	$\sigma_0,\sigma_1\in\Sigma[x]$,
	then evaluation at \(x=0\) gives \(\sigma_0(0)=0\). Since \(\sigma_0\)
	is a sum of squares, this implies $\sigma_0'(0)=0$.
	Differentiating the identity at \(x=0\) would yield \(1=0\), a
	contradiction. The same argument applies to \(-x\). Again \cite[Prop.~6.5.16]{S24} predicts the failure: at an isolated
	point of the positivity set, saturation requires two generators of order one
	whose product is nonpositive nearby, whereas here $p_1$ vanishes at the second order at zero.
\end{example}

\begin{lemma}[Approximate recovery of a natural generator]
	\label{lem:recover}
	Let $g\in\R[x]$ be such that
	\begin{equation}
		\label{eq:gpowerM}
		g^d\in\QM(p)
	\end{equation}
	for some odd integer $d\geq 1$.
	Then there exist constants
	$A>0$, $B\geq 0$, $C>0$
	depending only on $g$, $p$, and $d$, such that
	\begin{equation*}
		g+C\varepsilon_n\in\QM_{An+B}(p)
	\end{equation*}
	for every $n\geq 1$.
\end{lemma}

\begin{proof}
	Let $I=[-R,R]$, \text{where $R\geq 1$} is fixed by the Archimedean
	certificate of Lemma \ref{lem:archi}.  Since
	$I$ is compact and $g$ is a polynomial, there exists $C>0$ such that
	\begin{equation}
		\label{eq:Lchoice}
		g(x)\leq C
		\qquad\text{for every }x\in I.
	\end{equation}
	We apply Lemma~\ref{lem:chebatom} with $t=g(x)/C$.
	Multiplying \eqref{eq:chebidentity} by $C$ gives
	\begin{equation}
		\label{eq:recoveridentity}
		g+C\varepsilon_n
		=
		Cq_{n,d}(g/C) +
		C^{1-d}
		g^d
		W_n(g/C)^{d-1}.
	\end{equation}

	By \eqref{eq:Lchoice}, $g(x)/C\leq 1$ for every $x\in I$. Then, we deduce by
	\eqref{eq:qnnonneg} that $$Cq_{n,d}(g(x)/C)\geq 0 \quad \text{ for every } \quad
		x\in I.$$ Hence, by Lemma~\ref{lem:interval}
	(which we can apply since $R\geq 1$),
	\begin{equation*}
		Cq_{n,d}(g/C)\in\QM(R^2-x^2) \subset \QM(p) \quad
		\text{ using again Lemma~\ref{lem:archi}}.
	\end{equation*}

	For the second term, $d-1$ is even, so we can write
	$$C^{1-d}W_n(g/C)^{d-1}
		=
		(C^{(1-d)/2}
		W_n(g/C)^{(d-1)/2})^2.$$
	Thus \eqref{eq:gpowerM} and the closure of $\QM(p)$ under multiplication by
	squares imply that
	\begin{equation}
		\label{eq:secondtermM}
		C^{1-d}
		g^d W_n(g/C)^{d-1}\in\QM(p).
	\end{equation}
	Equations \eqref{eq:recoveridentity}-\eqref{eq:secondtermM} prove
	that $g  + C \varepsilon_n \in \QM(p)$.

	%
	%

	It remains to bound the degree of the resulting certificate in
	\(\QM(p)\).

	By \eqref{eq:chebdegree}, it holds that
	$$\deg q_{n,d}(g/C)
		\leq (\deg g)(1+n(d-1)).$$
	Since \(Cq_{n,d}(g/C)\) is nonnegative on \([-R,R]\) and \text{$R\geq 1$},
	Lemma~\ref{lem:interval} yields a representation
 {   \[
Cq_{n,d}(g/C)
\in
\QM_{s_n}(R^2-x^2),
\qquad
s_n:=
\left\lceil
\frac{(\deg g)(1+n(d-1))}{2}
\right\rceil.
\]
Thus there exist
\(\sigma_{0,n},\sigma_{1,n}\in\SOS\) such that
\[
Cq_{n,d}(g/C)
=
\sigma_{0,n}
+
\sigma_{1,n}(R^2-x^2),
\]
with
\[
\deg \sigma_{0,n}\leq 2s_n,
\qquad
\deg\sigma_{1,n}\leq 2s_n-2.
\]}
{Fix once and for all an Archimedean certificate
\[
R^2-x^2
=
\sum_{i=0}^{n_p}\tau_i p_i,
\qquad
\tau_i\in\SOS,
\]
and choose \(r_R\in\N\) such that
\[
R^2-x^2\in\QM_{r_R}(p).
\]
Equivalently,
\[
\deg(\tau_i p_i)\leq 2r_R,
\qquad i=0,\ldots,n_p.
\]
Substituting this fixed certificate into the preceding representation gives
\[
Cq_{n,d}(g/C)
=
\sigma_{0,n}
+
\sum_{i=0}^{n_p}
(\sigma_{1,n}\tau_i)p_i.
\]}
  Since a product of sums of squares is again a sum of squares, this is a
certificate in \(\QM(p)\). {Moreover,
\[
\deg\bigl((\sigma_{1,n}\tau_i)p_i\bigr)
\leq
2s_n-2+2r_R
\leq
2(s_n+r_R).
\]
Consequently,
\[
Cq_{n,d}(g/C)
\in
\QM_{s_n+r_R}(p).
\]
Since
\[
s_n
=
\left\lceil
\frac{\deg g+\deg g(d-1)n}{2}
\right\rceil
\leq
\left\lceil\frac{\deg g}{2}\right\rceil
+
n\left\lceil\frac{\deg g(d-1)}2\right\rceil,
\]
the order of this first certificate is bounded by an affine function of
\(n\).}  

{We now consider the second term in \eqref{eq:recoveridentity}.
Fix once and for all a certificate
\[
g^d
=
\sum_{i=0}^{n_p}\lambda_i p_i,
\qquad
\lambda_i\in\SOS,
\]
and choose \(r_g\in\N\) such that
\[
g^d\in\QM_{r_g}(p).
\]
Because \(d-1\) is even, the polynomial
\[
U_n(x)
:=
C^{(1-d)/2}
W_n(g(x)/C)^{(d-1)/2}
\]
is well defined, and
\[
C^{1-d}g^dW_n(g/C)^{d-1}
=
U_n^2g^d.
\]
Multiplying the fixed certificate of \(g^d\) by \(U_n^2\) yields
\[
C^{1-d}g^dW_n(g/C)^{d-1}
=
\sum_{i=0}^{n_p}(U_n^2\lambda_i)p_i.
\]
Since
\[
\deg U_n^2
=
D(n-1)(d-1),
\]
we obtain
\[
C^{1-d}g^dW_n(g/C)^{d-1}
\in
\QM_{\,r_g+\frac{D(d-1)}2(n-1)}(p).
\]
Thus the order of the second certificate is also bounded by an affine
function of \(n\).}

	Combining the certificates of the two terms in
	\eqref{eq:recoveridentity}, {there exist integers $A\geq1$, $B\geq0$},
	depending only on \(g\), \(p\), and \(d\), such that
	$g+C\varepsilon_n\in\QM_{An+B}(p)$ for every $n\geq1$.
\end{proof}

We can now prove Theorem~\ref{thm:main}.
\begin{proof}[Proof of Theorem~\ref{thm:main}]
	Define the  objective gap
	\begin{equation}
		\label{eq:fdef}
		q(x):=f(x)-f^\star.
	\end{equation}
	Then, it holds that $q\geq 0$ on $\K$, and by Lemma \ref{lem:natgen}  there
	exist fixed   $\sigma_i \in \SOS$
	such that $q = \sum_{i=0}^m \sigma_i g_i$ where the $g_i$'s for $i\geq 1$ are
	the natural generators associated to $\K$ for $1\leq i \leq m$ and $g_0:=1$.

	For every {$i=1,\ldots,m$}, Lemma~\ref{lem:oddpower}
	establishes that there
	exists an odd integer $d_i$ such that $g_i^{d_i}\in\QM(p)$.
	Applying Lemma~\ref{lem:recover}
	to each $g_i$, we
	deduce that there exist fixed constants $A_i>0$,  $B_i\geq 0$, $C_i > 0$ such
	that
	\begin{equation}\label{eq:gjdegree}
		g_i+C_i\varepsilon_n\in\QM_{A_in+B_i}(p).
	\end{equation}
	Now, observe that
	\[
		\sum_{i=0}^m \sigma_i (g_i+C_i\varepsilon_n) = q+\sigma\varepsilon_n
	\]
	where the polynomial $\sigma:=\sum_{i=0}^m C_i\sigma_i \in \SOS$ is fixed (since all $C_i$'s and $\sigma_i$'s are fixed).
	Also, since all $\sigma_i$ are fixed, the certificate degrees in
	\eqref{eq:gjdegree} imply that there exist constants $A>0$ and $B'\geq 0$
	such that
	\begin{equation}\label{eq:qs}
		q+\varepsilon_n \sigma \in \QM_{An+B'}(p).
	\end{equation}

	Let $C':=\max_{x\in[-R,R]}\sigma(x)$.
	Since $\sigma$ is a sum of SOS polynomials with positive coefficients, it
	holds that $C'\geq 0$.  Moreover $C'-\sigma(x)\geq 0$ for every $x\in[-R,R]$ by construction.
	By {Lemmas~\ref{lem:archi} and~\ref{lem:interval}},
	\begin{equation}
		\label{eq:C0S}
		C'-\sigma\in\QM(p).
	\end{equation}
	This is a fixed polynomial with a fixed certificate.  Multiplying
	\eqref{eq:C0S} by the positive scalar $\varepsilon_n$ and adding the result to \eqref{eq:qs} yields
	\begin{equation}
		\label{eq:finalcert}
		q+C'\varepsilon_n\in\QM_{An+B}(p)
	\end{equation}
	for some $B \geq B'$.

	Now let $r$ be sufficiently large and choose 
    {$n(r):=\left\lfloor\frac{r-B}{A}\right\rfloor$.
Then $n(r)\geq 1$ and $An(r)+B\leq r$.}
		Hence \eqref{eq:finalcert}  implies $q+C'\varepsilon_{n(r)}\in\QM_{r}(p)$.
		Recalling  \eqref{eq:fdef}, we obtain $f-\left(f^\star-C'\varepsilon_{n(r)}\right)
			\in\QM_{r}(p)$.
		By the definition \eqref{eq:sos} of $f_r$ as a supremum,
		$f_{r}\geq f^\star-C'\varepsilon_{n(r)}$.
		Therefore
		\begin{equation}
			\label{eq:errornr}
			0\leq f^\star-f_r
			\leq C'\varepsilon_{n(r)}
			\leq \frac{C'}{n(r)^2},
		\end{equation}
		where the last inequality follows from Lemma~\ref{lem:epsbound}.

{If \(r\geq2(A+B)\), then
\[
	n(r)
	\geq
	\frac{r-B}{A}-1
	=
	\frac{r-A-B}{A}
	\geq
	\frac{r}{2A}.
\]
Thus \eqref{eq:errornr} gives
\[
	0\leq f^\star-f_r
	\leq
	\frac{4C'A^2}{r^2}.
\]
Letting
\[
	C:=4C'A^2
\]
and increasing \(r\) if necessary proves \eqref{eq:mainrate}.}
\end{proof}

\section{Basic Example with Worst Convergence Rate}

Consider
\begin{equation}
	\label{eq:pop1}
	\tag{POP}
	f^\star:=\min_{x\in\R}\ f(x)
	\text{ s.t. }p(x)\ge0,
	\text{ with }
	f(x):=1-x,
	\quad
	p(x):=(1-x)^3(1+x).
\end{equation}
The feasible set is $K=[-1,1]$ and $f^\star=0$, attained
at $x=1$, where the constraint is cubically degenerate: $p$ has a
zero of multiplicity three at the minimizer.

For $r\ge2$, the order-$r$ relaxation error of \eqref{eq:pop1} is
\begin{equation}
	\label{eq:sos1}
	\tag{SOS$_r$}
	e_r
	:=
	f^* - \sup\{v\in\R:\ 1-x-v\in \QM_{r} (p)\bigr\}.
\end{equation}
Membership in the truncated quadratic module $\QM_{r}(p)$ certifies nonnegativity on $\K$, so
$e_r\ge f^\star=0$; and the value $0$ is never attained: if
$1-x=\sigma_0+p\sigma_1$ then $\sigma_0(1)=0$, so $\sigma_0$ vanishes
with even multiplicity at least two
while $p\sigma_1$ vanisheswith odd multiplicity at least three
whereas $1-x$ has a simple zero there.  Hence $e_r>0$ at every order $r$.

\begin{theorem}
	\label{thm:main1}
	For every $r\ge2$, it holds that \[e_r=\frac{1}{2r(r-1)}.\]
\end{theorem}

The proof occupies the rest of this section: one identity gives the
lower bound, three short lemmas and one Cauchy-Schwarz inequality give
the matching upper bound.

Let $T_r$ resp. $U_r$ be the Chebyshev polynomial of first resp. second kind, i.e.
$T_r(\cos\theta)=\cos r\theta$ and
$U_r(\cos\theta)\sin\theta=\sin(r+1)\theta$.

For a polynomial $f \in \R[x]$, $f'$ denotes its first order derivative.\\
Recall that both $T_r$ and $U_r$ have degree $r$ and that
$$T_r'=rU_{r-1}, \quad T_r=U_r-xU_{r-1} \quad \text{ and } \quad U_r=2xU_{r-1}-U_{r-2}.$$
Throughout, fix $r\ge2$ and set
\[
	A_r(x):= r\,T_{r-1}(x)-(r-1)\,T_r(x),
	\qquad
	B_r(x):=\frac{r\,U_{r-2}(x)-(r-1)\,U_{r-1}(x)}{1-x}.
\]

\begin{lemma}
	\label{lem:pell}
	It holds that $A_r$ and $B_r$ are polynomials with integer coefficients with $\deg A_r=r$,
	$\deg B_r=r-2$ and
	\begin{equation}
		\label{eq:pell}
		A_r(x)^2+p(x)\,B_r(x)^2=C_r(x):=1+2r(r-1)(1-x).
	\end{equation}
	Consequently $1-x+\frac{1}{2r(r-1)}\in\QM_{r}({p})$ and
	$e_r\le\frac{1}{2r(r-1)}$.
\end{lemma}

\begin{proof}
	Let $x:=\cos\theta$ and  $z:=\bigl(r-(r-1)e^{i\theta}\bigr)e^{i(r-1)\theta}$. Then
	\[
		|z|^2=r^2+(r-1)^2-2r(r-1)\cos\theta=1+2r(r-1)(1-x) = C_r(x).
	\]
	The real and imaginary parts of $z$ are
	\[
		\operatorname{Re}z
		=r\cos(r-1)\theta-(r-1)\cos r\theta
		=A_r(x)
	\]
	and
	\[
		\operatorname{Im}z
		=r\sin(r-1)\theta-(r-1)\sin r\theta
		=\sin\theta\,\bigl[r\,U_{r-2}(x)-(r-1)\,U_{r-1}(x)\bigr].
	\]
	The bracket vanishes at $x=1$, because
	$rU_{r-2}(1)-(r-1)U_{r-1}(1)=r(r-1)-(r-1)r=0$; dividing by the integer
	linear factor $1-x$ gives the integer coefficient polynomial $B_r(x)$ of degree $r-2$,
	and
	\[
		(\operatorname{Im}z)^2=(1-x^2)(1-x)^2B_r^2(x)=p(x)\,B_r^2(x).
	\]
	Then \eqref{eq:pell} writes
	\[
		(\operatorname{Re}z)^2+(\operatorname{Im}z)^2=|z|^2
	\]
	on $[-1,1]$, hence as polynomials.  Since $\deg A_r=r$
	(its leading coefficient is $-(r-1)2^{r-1}$), dividing \eqref{eq:pell} by
	$2r(r-1)$ exhibits $1-x+\frac{1}{2r(r-1)}$ as an element of
	$\QM_{r}(p)$ with admissible degrees.

	The fact that $\deg A_r = r$ and $\deg B_r = r-2$ are immediate from
		their definitions and $\deg T_r = r$ and $\deg U_r = r$.

\end{proof}

The table below displays the polynomials of identity \eqref{eq:pell} for small values of $r$.
\[
	\begin{array}{c|l|l|l}
		r & A_r(x)                         & B_r(x) & C_r(x) \\
		\hline
		2
		  & -2x^2+2x+1
		  & 2
		  & 5-4x
		\\[1mm]
		3
		  & -8x^3+6x^2+6x-3
		  & 8x+2
		  & 13-12x
		\\[1mm]
		4
		  & -24x^4+16x^3+24x^2-12x-3
		  & 24x^2+8x-4
		  & 25-24x
		\\[1mm]
		5
		  & -64x^5+40x^4+80x^3-40x^2-20x+5
		  & 64x^3+24x^2-24x-4
		  & 41-40x
	\end{array}
\]

\begin{lemma}
	\label{lem:roots}
	It holds that $A_r$ has $r$ simple real roots distributed as follows:
	\[-1<x_1<\cdots<x_{r-1}<1<x_r.\]
\end{lemma}

\begin{proof}
	Let $z_k:=\cos\frac{k\pi}{r-1}$ for $k=0,\ldots,r-1$, so that
	$1=z_0>\cdots>z_{r-1}=-1$.  Then $T_{r-1}(z_k)=(-1)^k$ and
	$T_r(z_k)=\cos (k\pi+\tfrac{k\pi}{r-1})=(-1)^k z_k$, so
	\[
		A_r(z_k)=(-1)^k \underbrace{(r-(r-1)z_k)}_{\geq 1}.
	\]
	The signs alternate, giving at least $r-1$ roots in $(-1,1)$.  Since
	$\deg A_r=r$ with negative leading coefficient and $A_r(1)=1>0$, there
	is a further root in $(1,\infty)$.  This accounts for $r=\deg A_r$
	roots, all simple.
\end{proof}

\begin{lemma}
	\label{lem:fo}
	With $c:= r(r-1)(2r-1)$,
	\begin{equation}
		\label{eq:fo}
		C_r(x)\,A_r'(x)+r(r-1)\,A_r(x)=c\,(1-x)^2\,B_r(x).
	\end{equation}
\end{lemma}

\begin{proof}
	From $T_r'=rU_{r-1}$, we deduce that
	$$A_r'=r(r-1)\bigl(U_{r-2}-U_{r-1}\bigr).$$
	So, after dividing by $r(r-1)$ and using
	$(1-x)B_r=rU_{r-2}-(r-1)U_{r-1}$,  \eqref{eq:fo} is
	equivalent,to
	$$C_r(x)\,(U_{r-2}(x)-U_{r-1}(x))+A_r(x)
		=(2r-1)(1-x)\bigl[r\,U_{r-2}(x)-(r-1)\,U_{r-1}(x)\bigr].$$
	By $T_r=U_r-xU_{r-1}$ and $U_r=2xU_{r-1}-U_{r-2}$, it holds that
	\begin{align*}
		A_r(x) & =r(U_{r-1}(x)-xU_{r-2}(x))-(r-1)(U_r(x)-xU_{r-1}(x)) \\
		       & =(r-(r-1)x)U_{r-1}(x)+(r-1-rx)U_{r-2}(x),
	\end{align*}
	so, in the left-hand side, the coefficient of $U_{r-2}$ is
	$$C_r(x)+r-1-rx=r(1-x)+2r(r-1)(1-x)=r(2r-1)(1-x)$$
	and the coefficient of $U_{r-1}$ is
	$$-C_r(x)+r-(r-1)x=(r-1)(1-x)-2r(r-1)(1-x)=-(r-1)(2r-1)(1-x).$$
	This is the right-hand side.
\end{proof}

\begin{lemma}
	\label{lem:weights}
	At every root $x_i$ of $A_r$ it holds
	\begin{equation*}
		p(x_i)\,B_r(x_i)\,A_r'(x_i)=c\,(1-x_i)^2>0 .
	\end{equation*}
\end{lemma}

\begin{proof}
	Differentiating \eqref{eq:pell} gives
	$2A_rA_r'+p'B_r^2+2pB_rB_r'=-2r(r-1)$.  At a common root of $A_r$ and
	$B_r$ the left side would vanish, a contradiction; so
	$B_r(x_i)\neq0$, and by \eqref{eq:pell} and Lemma~\ref{lem:roots},
	$C_r(x_i)=p(x_i)B_r(x_i)^2\neq0$. We apply Lemma ~\ref{lem:fo} by evaluating \eqref{eq:fo} at $x_i$ and
	multiplying by $p(x_i)B_r(x_i)/C_r(x_i)$, we obtain
	\[
		p(x_i)\,B_r(x_i)\,A_r'(x_i)
		=c\,(1-x_i)^2\,\frac{p(x_i)B_r(x_i)^2}{C_r(x_i)}
		=c\,(1-x_i)^2,
	\]
	which is positive because $x_i\neq1$ (see Lemma~\ref{lem:roots}.
\end{proof}

\begin{proof}[Proof of Theorem~\ref{thm:main1}]
	By Lemma~\ref{lem:weights}, the positive numbers
	\[
		\widehat w_i
		:=
		\frac{1}{c(1-x_i)^2},
		\qquad
		i=1,\ldots,r,
	\]
	satisfy
	\[
		\widehat w_i\,p(x_i)B_r(x_i)
		=
		\frac{1}{A_r'(x_i)}.
	\]
    Indeed, $x_i\neq 1$ by Lemma~\ref{lem:roots} and $c>0$, so each $\widehat w_i$
	is well defined and positive; and Lemma~\ref{lem:weights} asserts
	$p(x_i)B_r(x_i)A_r'(x_i)=c(1-x_i)^2\neq 0$, so that the three factors
	$p(x_i)$, $B_r(x_i)$, $A_r'(x_i)$ are all non-zero and dividing that identity
	by $c(1-x_i)^2A_r'(x_i)$ gives the claim.
	Let
	\[
		w_i
		:=
		\frac{\widehat w_i}{\sum_{j=1}^r\widehat w_j}
	\]
	and define the linear functional
	\[
		\ell(q)
		:=
		\sum_{i=1}^r w_i q(x_i),
		\qquad
		q\in\R[x]_{2r}.
	\]
	Since the weights \(w_i\) are positive and sum to one, we have $\ell(1)=1$
	and $$\ell(q^2)= \sum_{i=1}^r w_i q(x_i)^2\geq 0$$ for all $q\in\R[x]$ such that  \(\deg q\leq r\).

	We next prove that $\ell(pq^2)\geq 0$ for every
		$q\in \R_{r-2}[x]$ (recall that $p$ has degree $4$).
	Let \(q\in\R[x]\) satisfy \(\deg q\leq r-2\).
	The Lagrange interpolation formula at the \(r\) simple roots
	\(x_1,\ldots,x_r\) of \(A_r\) reads
	\[
		q(x)
		=
		\sum_{i=1}^r
		q(x_i)
		\frac{A_r(x)}{A_r'(x_i)(x-x_i)}.
	\]
	Since the roots of $A_r$ are simple, each quotient $A_r(x)/(x-x_i)$ is a
	polynomial of degree $r-1$ whose leading coefficient is that of $A_r$, namely
	$\alpha_r:=-(r-1)2^{r-1}\neq 0$. The coefficient of $x^{r-1}$ in the
	right-hand side is therefore $\alpha_r\sum_{i=1}^rq(x_i)/A_r'(x_i)$, whereas
	it vanishes in the left-hand side because $\deg q\leq r-2$. Dividing by
	$\alpha_r$ we obtain
	\[
		\sum_{i=1}^r\frac{q(x_i)}{A_r'(x_i)}=0.
	\]
Let $S:=\sum_{j=1}^r\widehat w_j$, a positive real number, so that
	$w_i=\widehat w_i/S$ for every $i$. Dividing by $S$ the identity
	$\widehat w_ip(x_i)B_r(x_i)=1/A_r'(x_i)$ established above yields
	\[
		w_i\,p(x_i)B_r(x_i)=\frac{1}{S\,A_r'(x_i)},
		\qquad i=1,\ldots,r,
	\]
	so that the weighted sum against $B_r$ is, up to the factor $1/S$, the sum of
	the Lagrange coefficients just computed:
	\begin{equation}
		\label{eq:orth}
		\sum_{i=1}^r
		w_i p(x_i)B_r(x_i)q(x_i)
		=
		\frac{1}{S}\sum_{i=1}^r\frac{q(x_i)}{A_r'(x_i)}
		=
		0
	\end{equation}
	for every $q\in\R[x]$ with $\deg q\leq r-2$.
    By Lemma~\ref{lem:roots}, $-1 < x_i < 1$ for $1\leq i \leq r-1$ and $x_r > 1$. Then, it holds that
	$p(x_i)>0$,
	$i=1,\ldots,r-1$,
	whereas \(p(x_r)<0\). Set
	\begin{equation}\label{eq:lambda}
		\lambda
		:=
		-w_rp(x_r)>0.
	\end{equation}
	We introduce the weighted scalar product on the values at the
	interior nodes \(x_1,\ldots,x_{r-1}\):
	\[
		\langle u,v\rangle_{\mathrm{int}}
		:=
		\sum_{i=1}^{r-1}
		w_i p(x_i)u(x_i)v(x_i).
	\]
	All the weights \(w_i p(x_i)\), \(i=1,\ldots,r-1\), are positive.
	Isolating the exterior node \(x_r\) in \eqref{eq:orth} gives
	\begin{equation}
		\label{eq:reproducing}
		\langle B_r,q\rangle_{\mathrm{int}}
		=
		\lambda B_r(x_r)q(x_r).
	\end{equation}
	Since \(\deg B_r=r-2\), we may choose \(q=B_r\) in
	\eqref{eq:reproducing}. We obtain
	\begin{equation}
		\label{eq:Brnorm}
		\langle B_r,B_r\rangle_{\mathrm{int}}
		=
		\lambda B_r(x_r)^2.
	\end{equation}

	Let now \(q\in\R[x]_{r-2}\) be arbitrary. The Cauchy-Schwarz
	inequality for the weighted scalar product gives
	\[
		\langle B_r,q\rangle_{\mathrm{int}}^2
		\leq
		\langle B_r,B_r\rangle_{\mathrm{int}}
		\langle q,q\rangle_{\mathrm{int}}.
	\]
	Using \eqref{eq:reproducing} and \eqref{eq:Brnorm}, this becomes
	\[
		\lambda^2 B_r(x_r)^2q(x_r)^2
		\leq
		\lambda B_r(x_r)^2
		\sum_{i=1}^{r-1}
		w_i p(x_i)q(x_i)^2.
	\]
	By Lemma~\ref{lem:weights}, \(B_r(x_r)\neq0\). Since also
	\(\lambda>0\), division by \(\lambda B_r(x_r)^2\) yields
	\[
		\lambda q(x_r)^2
		\leq
		\sum_{i=1}^{r-1}
		w_i p(x_i)q(x_i)^2.
	\]
	Recalling \eqref{eq:lambda}, we conclude that
	\[
		\ell(pq^2)
		=
		\sum_{i=1}^r
		w_i p(x_i)q(x_i)^2
		=
		\sum_{i=1}^{r-1}
		w_i p(x_i)q(x_i)^2
		-\lambda q(x_r)^2
		\geq0.
	\]
	In particular, taking \(q=B_r\) in \eqref{eq:Brnorm} gives
	\begin{equation}\label{eq:gbr}
		\ell(pB_r^2)
		=
		\sum_{i=1}^{r-1}
		w_i p(x_i)B_r(x_i)^2
		-\lambda B_r(x_r)^2
		=
		0.
	\end{equation}
	Applying the linear functional $\ell$ to the identity \eqref{eq:pell} yields
	\[
		\ell(A_r^2)+\ell(pB_r^2)=\ell(C_r)=1+2r(r-1)\ell(1-x)=0
	\]
	since $\ell(A_r^2)=\sum_{i=1}^r w_i A_r(x_i)^2 = 0$ because $x_i$ are the roots of $A_r$, and $\ell(pB_r^2)=0$ by \eqref{eq:gbr}. Therefore
	\begin{equation}\label{eq:elx}
		\ell(1-x) = -\frac{1}{2r(r-1)}.
	\end{equation}

	We have thus shown that $\ell(q^2)\geq0$ for every $q\in\R[x]_r$, and $\ell(pq^2)\geq0$ for every $q\in\R[x]_{r-2}$, i.e. $\ell$ is nonnegative on the truncated quadratic
		module $\QM_{r}(p)$ and normalized by $\ell(1)=1$.
		Let $v\in\mathbb{R}$ be any feasible value for the SOS problem \eqref{eq:sos1}.	Since $\ell$ is nonnegative on $\QM_{r}(p)$, it follows that
	$0
		\leq
		\ell(1-x-v)
		=
		\ell(1-x)-v\ell(1)
		=
		\ell(1-x)-v$ and hence $v\leq\ell(1-x)$
	for every feasible value $v$. Taking the supremum over all feasible
	$v$ and recalling \eqref{eq:elx} yields
	\[
		e_r \geq -\ell(1-x) = \frac{1}{2r(r-1)}.
	\]
	On the other hand, Lemma~\ref{lem:pell} gives the reverse inequality, and
	consequently
	\[
		e_r
		=
		\frac{1}{2r(r-1)}.
	\]
\end{proof}

\begin{remark}[Comparison with the Stengle problem]\label{rk:comparison}
	The classical symmetric Stengle problem is $\min_{x\in\R}\{\,1-x^2:\ (1-x^2)^3\ge0\,\}.$ Its feasible set is again $[-1,1]$ and in \cite{H25} it is shown that its exact order-$r$ moment-SOS error
	is $\frac1{r(r-2)}$ for $r\ge3$. The structural reason for $e_r>0$ at every order is that $\QM((1-x^2)^3)$ is
	not stable. Stengle's original estimate, recorded as \cite[Ex.~6.6.6(3)]{S24}, states that there is $c>0$ with $\deg\sigma_0>c\,\varepsilon^{-1/2}$
	in every identity $1-x^2+\varepsilon=\sigma_0+\sigma_1(1-x^2)^3$,
	$\sigma_0,\sigma_1\in\SOS$; the exact value $\frac1{r(r-2)}$ pins down the
	constant in this $\varepsilon^{-1/2}$ law.
	The proof of Theorem~\ref{thm:main1} follows the same broad primal-dual
	architecture as the exact analysis of  \cite{H25}, but the asymmetric
	constraint and the linear objective function makes the analysis simpler.
\end{remark}

\section{Univariate and Bivariate Cusp Problems}\label{sec:cusp}

In this section we show that the extremal example \eqref{eq:pop1} is,
up to an affine change of variables, equivalent to a univariate
problem with split constraints, which is in turn an exact univariate
reduction of a bivariate cusp problem studied in the literature
\cite[Ex.~9.4.6(3)]{M08}, \cite[Ex.~A.3]{BM25}, \cite{K25}. As a consequence, the exact
$\Theta(1/r^2)$ convergence rate of Theorem~\ref{thm:main1} transfers
to all these formulations.

\subsection{Affine change of variables}

The affine substitution $x\mapsto 1-2x$ maps the interval $[0,1]$ onto
$[-1,1]$ and transforms \eqref{eq:pop1} into
\begin{equation}
	\label{eq:pop1p}
	\tag{POP$'$}
	\min_{x\in\R}\ x
	\quad\text{s.t.}\quad
	x^3(1-x)\ge0,
\end{equation}
whose feasible set is $[0,1]$ and whose value is $0$, attained at
$x=0$, where the constraint has a zero of multiplicity three.
Indeed, substituting $x\mapsto 1-2x$ in \eqref{eq:pop1} gives objective
$2x$ and constraint $(2x)^3(2-2x)=16x^3(1-x)$, and neither the
positive scaling of the constraint nor the positive scaling of the
objective affects membership in the truncated quadratic module at a
given order. Consequently, denoting by $e'_r$ the order-$r$ relaxation
error of \eqref{eq:pop1p}, Theorem~\ref{thm:main1} yields
\[
	e'_r=\frac{e_r}{2}=\frac{1}{4r(r-1)},
	\qquad r\ge 2.
\]
\begin{remark}\label{rem:exercise}
	The constraint polynomial $x^3(1-x)$ of \eqref{eq:pop1p} is the subject of
	\cite[Exerc.~5.5.5, p.~200]{S24}, where the reader is asked to prove that in
	any sequence of representations
	$x+\tfrac1n=\sigma_{0,n}+\sigma_{1,n}\,x^3(1-x)$ with
	$\sigma_{0,n},\sigma_{1,n}\in\SOS$ one necessarily has
	$\deg\sigma_{0,n},\deg\sigma_{1,n}\to\infty$. The value
	$e'_r=\frac{1}{4r(r-1)}$ makes this quantitative: such a representation exists
	within $\QM_r(x^3(1-x))$ if and only if $4r(r-1)\geq n$, so the smallest
	admissible order is $r_n=\big\lceil\tfrac{1+\sqrt{1+n}}{2}\big\rceil$ and the
	certificate degrees grow exactly like $\Theta(\sqrt{n})$. The same exercise,
	part (a), shows that $\QM_r(x^3(1-x))$ is a closed cone, so the supremum in
	\eqref{eq:sos1} is attained; here this is also witnessed explicitly by
	Lemma~\ref{lem:pell}.
\end{remark}
\subsection{Splitting the constraint}

The constraint polynomial of \eqref{eq:pop1p} is the product of the two
polynomials $x^3$ and $1-x$. Splitting it accordingly yields the
univariate problem
\begin{equation}
	\label{eq:pop2p}
	\tag{POP$''$}
	\min_{x\in\R}\ x
	\quad\text{s.t.}\quad
	x^3\ge0,\quad 1-x\ge0,
\end{equation}
with the same feasible set $[0,1]$ and the same value $0$. Although
the truncated quadratic modules of the two descriptions differ at each
fixed order, the underlying quadratic modules coincide, and the
truncations agree up to a unit shift of the order.

\begin{lemma}[Product versus split generators]\label{lem:split}
	It holds that
	\[
		\QM(x^3(1-x))=\QM(x^3,\,1-x)
	\]
	and, for every $r\ge2$,
	\[
	\QM_{r}(x^3(1-x))\subseteq\QM_{r+1}(x^3,\,1-x),
	\qquad
	\QM_{r}(x^3,\,1-x)\subseteq\QM_{r+1}(x^3(1-x)).
	\]
\end{lemma}

\begin{proof}
	The three polynomial identities
	\begin{equation}\label{eq:splitid}
		\begin{aligned}
			x^3      & = x^4 + x^3(1-x),                                         \\
			1-x      & = (1-x)^2\left((x+\tfrac12)^2+\tfrac34\right) + x^3(1-x), \\
			x^3(1-x) & = x^4(1-x) + (1-x)^2x^3,
		\end{aligned}
	\end{equation}
	are verified by expansion. The first two show that both generators
	$x^3$ and $1-x$ belong to $\QM(x^3(1-x))$, hence
	$\QM(x^3,1-x)\subseteq\QM(x^3(1-x))$; the third shows that the product
	generator $x^3(1-x)$ belongs to $\QM(x^3,1-x)$, hence the reverse
	inclusion. This proves the equality of the quadratic modules.

	For the truncated statements, consider an element
	$\sigma_0 +\sigma_1 \cdot  x^3(1-x)\in\QM_{{r}}(x^3(1-x))$, so that
	$$\deg\sigma_0\leq 2r \quad \text{ and }\quad
		\deg\sigma_1\leq 2r-4.$$ The third identity
	of \eqref{eq:splitid} gives
	\[
		\sigma_1 \cdot x^3(1-x)
		=(\sigma_1 \cdot (1-x)^2)  x^3+(\sigma_1 \cdot x^4) (1-x),
	\]
	where $\sigma_1 \cdot (1-x)^2$ and $\sigma_1 \cdot x^4$ are SOS of degrees at most
	$2r-2$ and $2r$, so that
	$$\deg(\sigma_1 \cdot (1-x)^2  x^3)\leq 2r+1\quad
		\text{ and }
		\quad \deg(\sigma_1 \cdot x^4 (1-x))\leq 2r+1.$$
	Hence $\sigma_0+\sigma_1 \cdot x^3(1-x)\in\QM_{r+1}(x^3,1-x)$.

	Conversely, consider
	$\sigma_0+\sigma_1 \cdot x^3+\sigma_2 \cdot (1-x)\in\QM_{r}(x^3,1-x)$, with
	$\deg\sigma_1\leq 2r-4$ and $\deg\sigma_2\leq 2r-2$. By the first two
	identities of \eqref{eq:splitid},
	\[
		\sigma_1 \cdot x^3+\sigma_2 \cdot (1-x)
		=
		\sigma_1 \cdot x^4+\sigma_2 \cdot (1-x)^2\left((x+\tfrac12)^2+\tfrac34\right)
		+(\sigma_1+\sigma_2)\,x^3(1-x),
	\]
	where the first two terms are SOS of degrees at most $2r$ and $2r+2$,
	and the multiplier $\sigma_1+\sigma_2$ is SOS of degree at most $2r-2$,
	so that $\deg((\sigma_1+\sigma_2)x^3(1-x))\leq 2r+2$. Hence the element
	belongs to $\QM_{r+1}(x^3(1-x))$.
\end{proof}

Denoting by $e''_r$ the order-$r$ relaxation error of \eqref{eq:pop2p},
Lemma~\ref{lem:split} immediately yields a two-sided estimate.

\begin{corollary}[Convergence rate of the split formulation]\label{cor:cusprate}
	For every $r\ge3$, it holds
	\[
		\frac{1}{4r(r+1)}
		\;=\;e'_{r+1}\;\leq\;e''_r\;\leq\;e'_{r-1}\;=\;
		\frac{1}{4(r-1)(r-2)}.
	\]
	In particular
	\[
		\lim_{r\to\infty}r^2e''_r=\lim_{r\to\infty}r^2e'_r=\frac14,
	\]
	so the moment-SOS hierarchies of \eqref{eq:pop1p} and \eqref{eq:pop2p}
	share the same exact asymptotic convergence rate $\Theta(1/r^2)$.
\end{corollary}

\begin{proof}
	By Lemma~\ref{lem:split}, every certificate for \eqref{eq:pop2p} at
	order $r$ is a certificate for \eqref{eq:pop1p} at order $r+1$, so
	$e'_{r+1}\leq e''_r$, and symmetrically $e''_r\leq e'_{r-1}$. The
	limits follow from $e'_r=\frac{1}{4r(r-1)}$.
\end{proof}

Thus splitting the quartic constraint of \eqref{eq:pop1p} into a cubic
and a linear constraint changes the relaxation values only marginally:
the cost of separating the generators is at most one relaxation order.

\subsection{Reduction of the bivariate cusp problem}

Consider now the bivariate cusp problem
\begin{equation}
	\label{eq:popcusp}
	\tag{POP$^{\rm cusp}$}
	\min_{(x_1,x_2)\in\R^2}\ x_1
	\quad\text{s.t.}\quad
	x_1^3-x_2^2\ge0,\quad 1-x_1\ge0,
\end{equation}
whose feasible set $K$ is delimited by the cusp arcs
$x_2=\pm x_1^{3/2}$ and the vertical segment $x_1=1$. This problem appears in
\cite[Ex.~9.4.6(3)]{M08} and \cite[Ex.~A.3]{BM25} as an elementary
instance for which the \L{}ojasiewicz exponent of the constraint
description equals $3$, and it is used as a benchmark for high
precision semidefinite solvers in \cite{K25}. Its value is $0$,
attained at the cusp point $(0,0)$, where neither constraint
qualification nor strict complementarity holds.

The univariate problem \eqref{eq:pop2p} is precisely the restriction
of \eqref{eq:popcusp} to the axis $x_2=0$, with $x:=x_1$. The
following lemma shows that this restriction is exact order by order
for the SOS hierarchy.

\begin{lemma}[Exact univariate reduction]\label{lem:bivred}
	For every $r\ge2$, the order-$r$ SOS relaxation values of
	\eqref{eq:popcusp} and \eqref{eq:pop2p} coincide, i.e. denoting by
	$e^{\rm cusp}_r$ the order-$r$ relaxation error of \eqref{eq:popcusp},
	it holds
	\[
		e^{\rm cusp}_r=e''_r.
	\]
\end{lemma}

\begin{proof}
	If $x_1-v=\sigma_0+\sigma_1 \cdot (x_1^3-x_2^2)+\sigma_2 \cdot (1-x_1)$ is a
	certificate of order $r$ for \eqref{eq:popcusp}, with
	$\sigma_i\in\Sigma[x_1,x_2]$ of admissible degrees, then restricting
	to $x_2=0$ gives
	$$x_1-v=\sigma_0(x_1,0)+\sigma_1(x_1,0)x_1^3+\sigma_2(x_1,0)(1-x_1),$$
	where the restricted multipliers are univariate SOS of the same
	degrees, i.e. a certificate of order $r$ for \eqref{eq:pop2p}.

	Conversely, if $x-v=\sigma_0+\sigma_1 \cdot x^3+\sigma_2 \cdot (1-x)$ is a
	certificate of order $r$ for \eqref{eq:pop2p}, then substituting
	$x:=x_1$ and using the identity
	$x_1^3=(x_1^3-x_2^2)+x_2^2$ gives
	\[
		x_1-v=(\sigma_0(x_1)+x_2^2\sigma_1(x_1))
		+\sigma_1(x_1)(x_1^3-x_2^2)+\sigma_2(x_1)(1-x_1),
	\]
	where $\sigma_0(x_1)+x_2^2\sigma_1(x_1)$ is a bivariate SOS polynomial
	of degree at most
	$$\max\{\deg\sigma_0,\deg\sigma_1+2\}\leq 2r,$$ since
	$\deg\sigma_1\leq 2r-4$. This is a certificate of order $r$ for
	\eqref{eq:popcusp}. The two relaxation values therefore coincide at
	every order $r\ge2$.
\end{proof}

Combining Corollary~\ref{cor:cusprate} and Lemma~\ref{lem:bivred}, the
bivariate cusp problem \eqref{eq:popcusp} inherits the exact
asymptotic convergence rate
\[
	\lim_{r\to\infty}r^2e^{\rm cusp}_r=\frac14
\]
of the extremal example \eqref{eq:pop1}. This is consistent with the
general \L{}ojasiewicz-based bounds of \cite{BM23,BMP25}: the cusp
description has \L{}ojasiewicz exponent $\mathrm{L}=3$, yet the actual
convergence rate does not deteriorate below the universal univariate
exponent $2$ of Theorem~\ref{thm:main}. Figure~\ref{fig:cuspmom}
displays the outer approximations of $K$ computed by the
moment relaxations of orders $r=2,3,4$, illustrating how the
relaxations overshoot the cusp point along the negative $x_1$ axis by
the amount $e^{\rm cusp}_r$.

\begin{figure}[ht]
	\centering
	\includegraphics[width=.6\textwidth]{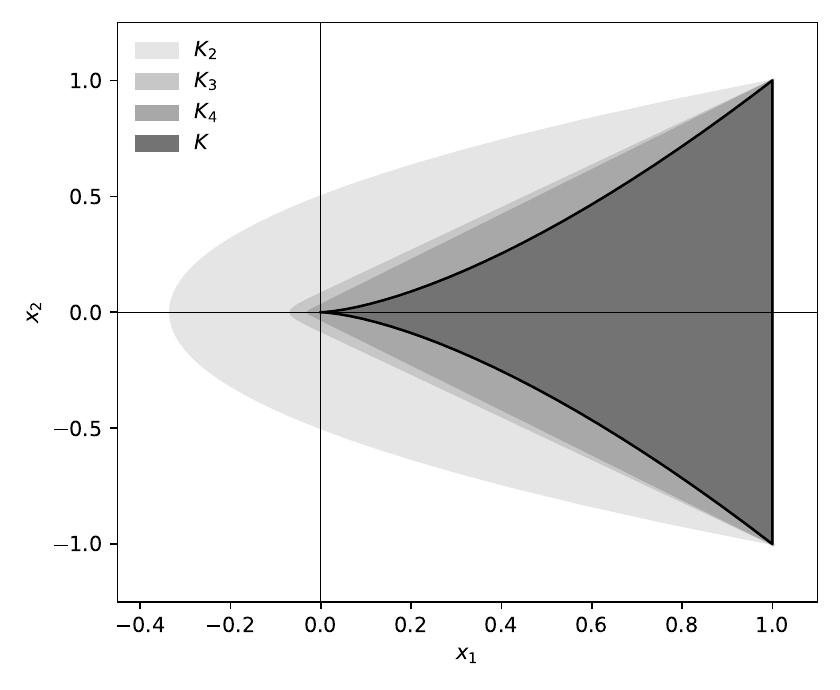}
	\caption{Outer approximations of the cusp feasible set
	$K$ (dark gray) computed by the moment relaxations of
	orders $r=2,3,4$ (nested regions, from light to darker). The
	minimum of $x_1$ over the order-$r$ relaxation equals
	$-e^{\rm cusp}_r$, i.e. $-\tfrac13$, $-\tfrac1{15}$ and
	$-\tfrac1{35}$ for $r=2,3,4$ respectively.}
	\label{fig:cuspmom}
\end{figure}

\begin{remark}[Exact relaxation values]\label{rem:exact}
	The two-sided estimate of Corollary~\ref{cor:cusprate} can be refined
	to an exact value. The substitution $x=1-t^2$ maps the univariate cusp
	problem \eqref{eq:pop2p} to the symmetric Stengle problem
	$\min_t\{1-t^2:(1-t^2)^3\ge0\}$ recalled in Remark \ref{rk:comparison}, and a parity argument on the extremal
	Chebyshev-Gegenbauer certificates of \cite{H25} shows that the order-$r$
	relaxation of \eqref{eq:pop2p} matches the order-$(2r-1)$ relaxation
	of the Stengle problem. Together with the exact Stengle value
	$\frac{1}{s(s-2)}$ at order $s$ from \cite{H25}, this yields
	\[
		e^{\rm cusp}_r=e''_r=\frac{1}{(2r-1)(2r-3)},
		\qquad r\ge2,
	\]
	which indeed satisfies the bounds of Corollary~\ref{cor:cusprate}.
\end{remark}

\section{Conclusion}

We have established a universal quadratic upper bound on the convergence rate of the moment-SOS hierarchy for polynomial optimization on a bounded subset of the real line. More precisely, for every fixed univariate POP of the form \eqref{eq:pop}, there exists a constant $C>0$ such that
$0\leq e_r\leq \frac{C}{r^2}$
for all sufficiently large relaxation orders $r$, where $e_r$ is the error at the relaxation order $r$. We have also shown that this quadratic rate is optimal, and achieved by the elementary degree-four example \eqref{eq:pop1} for which $e_r=1/(2r(r-1))$, $r\geq 2$.

The main difficulty is that the natural generators of the feasible set need not belong to the quadratic module generated by the original constraints. We overcome this difficulty by combining the univariate structure theorem for quadratic modules with a Chebyshev polynomial construction. Odd powers of the natural generators belong to the original quadratic module, and the Chebyshev construction approximately recovers the generators themselves with an error of order $1/r^2$ and with certificate degree growing linearly in $r$.

A natural open question is whether the quadratic exponent remains universal for bivariate polynomial optimization. At present, we do not know any bivariate POP for which the convergence of the moment-SOS hierarchy is slower than $O(1/r^2)$. It would therefore be desirable to prove that no such example exists, namely that, under the usual Archimedean assumptions and for fixed problem data, every bivariate POP satisfies $e_r=O(1/r^2)$.
Such a result would show that the deterioration predicted by general dimension-dependent and \L{}ojasiewicz-based estimates is not intrinsic in dimension two, and would constitute a first step toward identifying the precise role of dimension in the worst-case convergence rate of the moment-SOS hierarchy.

\appendix

\section*{Acknowledgement}

The mathematical writing was done with the help of ChatGPT 5.6 and Claude Fable 5.
This work benefited from feedback by Victor Magron.
The first author is grateful to Juan Vera Lizcano for figuring out, during the Oberwolfach workshop on conic linear optimization for computer-assisted proofs \cite{OWR26}, the reformulation of the univariate cusp problem as a Stengle problem.

\end{document}